\documentclass{amsart}
\usepackage{amssymb,latexsym}
\usepackage{amscd,amsthm}

\usepackage[all]{xy}
\usepackage{tikz}

\usepackage{tikz-cd}

\newtheorem{theorem}{Theorem}[section]

\newtheorem{lemma}[theorem]{Lemma}
\newtheorem{proposition}[theorem]{Proposition}
\newtheorem{corollary}[theorem]{Corollary}

\theoremstyle{definition}
\newtheorem{definition}[theorem]{Definition}

\newtheorem*{notation}{Notation}

\newtheorem*{theorem*}{Theorem}

\DeclareMathOperator{\Ext}{Ext}
\DeclareMathOperator{\Hom}{Hom}

\DeclareMathOperator{\Ker}{Ker}

 \newcommand{\Oplus}{\ensuremath{\vcenter{\hbox{\scalebox{1.2}{$\oplus$}}}}} 				
 
   \newcommand{\mycap}{\ensuremath{\vcenter{\hbox{\scalebox{1.4}{$\,\cap\,$}}}}}             
 
\newcommand{\class}[1]{\mathcal{#1}}   

\newcommand{\ch}{\textnormal{Ch}(R)}

\newcommand{\rmod}{R\textnormal{-Mod}}

\newcommand{\tilclass}[1]{\widetilde{\class{#1}}}
\newcommand{\dgclass}[1]{dg\widetilde{\class{#1}}}

\newcommand{\rightperp}[1]{#1^{\perp}}
\newcommand{\leftperp}[1]{{}^\perp #1}

\begin{document}

\title{Weakly Ding injective complexes}

\author{James Gillespie}
\address{J.G. \ Ramapo College of New Jersey \\
         School of Theoretical and Applied Science \\
         505 Ramapo Valley Road \\
         Mahwah, NJ 07430\\ U.S.A.}
\email[Jim Gillespie]{jgillesp@ramapo.edu}
\urladdr{http://pages.ramapo.edu/~jgillesp/}

\author{Alina Iacob}
\address{A.I. \ Department of Mathematical Sciences \\
         Georgia Southern University \\
         Statesboro (GA) 30460-8093 \\ U.S.A.}
\email[Alina Iacob]{aiacob@GeorgiaSouthern.edu}
\urladdr{https://sites.google.com/a/georgiasouthern.edu/aiacob/home}

\date{\today}


\begin{abstract} Working over a  (left) coherent ring, we consider the class of weakly Ding injective complexes.
These are the cycles of the exact complexes of FP-injective complexes that stay exact when applying   $\Hom(A,-)$ for any FP-injective complex $A$. We study the  cotorsion pair generated by the class of all such  complexes, and exhibit it as part of an abelian model structure. As an application we show that when $R$ is a Ding-Chen ring, its stable chain complex category is compactly generated and triangle equivalent to the stable category of four Frobenius categories. They are the categories of all (i) complexes of Gorenstein injective modules, (ii) complexes of Gorenstein projective modules, (iii) complexes of Gorenstein flat-cotorsion modules, and (iv)  complexes of Gorenstein FP-pro-injective modules.
\end{abstract}

\maketitle

\section{introduction}
Let $R$ be a ring. Going back to~\cite{stenstrom-fp}, an $R$-module $A$ is  called FP-injective if $\Ext^1_R(F,A) = 0$ for all finitely presented modules $F$. When $R$ is a coherent ring,  the class of FP-injective modules  possesses several   properties that are well-known to hold for injective modules over Noetherian rings; for example, closure under direct limits and closure under taking cokernels of monomorphisms. Along the same lines, Ding injective $R$-modules were introduced in~\cite{ding and mao 08, gillespie-Ding-Chen rings} as a generalization of the Gorenstein injective modules but with good properties over coherent rings, especially over the Ding-Chen rings of~\cite{ding and chen 93}. Such rings are the coherent analog of  Iwanaga-Gorenstein rings;  they are two-sided coherent rings $R$ having finite self FP-injective dimensions,  $FP\textnormal{-id}(R_R)<\infty$ and $FP\textnormal{-id}({}_RR)<\infty$.  By definition, Ding injective modules are  the cycles of  exact complexes of injective modules that stay exact when applying the functor $\Hom_R(A,-)$ for any FP-injective module $A$. 
Note that instead of exact complexes of injectives, one could also consider exact complexes of FP-injective modules.  In particular,  a module $M$ is said to be  \emph{weakly Ding injective} if it equals  a  cycle of some exact complex of FP-injective modules that remains exact when applying the functor $\Hom_R(A,-)$ for any FP-injective module $A$.  The weakly Ding injective modules have been studied recently in~\cite{Goren-FP-inj, iacob-weakly-Ding-1, iacob-weakly-Ding-2, iacob-weakly-Ding-3, gillespie-FP-pro-injective}.  
One interesting aspect of these modules is that, over coherent rings, they generate a class of modules possessing a number of homological  properties that are dual to that of the Gorenstein flat modules of~\cite{G-flat-modules, enochs-jenda-book, saroch-stovicek-G-flat}
 
In this paper we define and study  \emph{weakly Ding injective complexes} of $R$-modules.  Indeed the above definitions make perfect sense in the category $\ch$ of chain complexes of $R$-modules.
So we define a weakly Ding injective complex to be a complex that  arises as a cycle of some exact complex of FP-injective complexes which remains exact when applying the functor $\Hom_{\ch}(A, -)$  for any FP-injective complex $A$. In Section~\ref{sec-results} we show that over a (left) coherent ring $R$,  a number of properties of  weakly Ding injective $R$-modules extend to complexes of $R$-modules. Firstly, any weakly Ding injective complex is  just a direct sum of some Ding injective complex with an FP-injective complex.
 Using this, we show in Corollary~\ref{corollary-left-perp} that ${}^\bot \textnormal{w}\mathcal{DI}_{\textnormal{C}}$, the left $\Ext$-orthogonal class to that of all  weakly Ding injective complexes, identifies with the class of all complexes of FP-projective modules in  ${}^\bot \textnormal{C}(\mathcal{DI})$,  the left $\Ext$-orthogonal  class  to that of all complexes of Ding injective modules. That is, 
$${}^\bot \textnormal{w}\mathcal{DI}_{\textnormal{C}} = \textnormal{C}(\mathcal{FP})\mycap {}^\bot \textnormal{C}(\mathcal{DI}).$$

Our interest in the weakly Ding injective complexes stems from the fact that  they  generate a cotorsion pair that is cogenerated by a set, and hence complete. In fact, we show in Theorem~\ref{them-cot-pair} that the weakly Ding injective complexes generate the class of fibrant objects in an hereditary abelian model structure on $\ch$. That is, the fibrant objects are those in  $({}^\bot\textnormal{w}\mathcal{DI}_{\textnormal{C}})^\bot$,   the left-right-perp, again relative to $\Ext^1_{\ch}(-,-)$,  to the class of all weakly Ding injective complexes.
We show in Proposition~\ref{prop-left-right-perp} that this class  coincides with  the class of all complexes having components in the left-right-perp, now relative to $\Ext^1_R(-,-)$,   to the class of all  weakly Ding injective modules. In the notation we will be using, 
$$({}^\bot\textnormal{w}\mathcal{DI}_{\textnormal{C}})^\bot   = \textnormal{C}[({}^\bot\textnormal{w}\class{DI})^\bot],$$ 
and so we have an hereditary abelian model structure on $\ch$:
   \begin{equation}\tag{$\dagger$}\label{eq-intro}
   \mathfrak{M} =  (\textnormal{C}(\mathcal{FP}) , {}^\bot \textnormal{C}(\mathcal{DI}) ,   \textnormal{C}[({}^\bot{\textnormal{w}\mathcal{DI}})^\bot]).
   \end{equation}

Section~\ref{section-Ding-Chen} shows how the model structure in  (\ref{eq-intro}) provides  a new representation of the stable category of chain complexes of modules over a Ding-Chen ring $R$.
By definition, this category  is 
$$\textnormal{StCh}(R) := \ch/\tilclass{W},$$
the triangulated localization  of $\ch$  (in the sense of~\cite[\S6.7]{gillespie-book}), 
by the full subcategory $\tilclass{W}\subseteq \ch$ consisting of all complexes of finite flat dimension. 
It is  already known, and we explain this in Section~\ref{section-Ding-Chen},  that $\textnormal{StCh}(R)$ exists and has two explicit representations that are dual to one another. These are  
$$\textnormal{St}(\textnormal{C}(\class{GI})) \simeq\ \textnormal{StCh}(R) \simeq \textnormal{St}(\textnormal{C}(\class{GP})),$$
where $\textnormal{St}(\textnormal{C}(\class{GI}))$ denotes (resp. $ \textnormal{St}(\textnormal{C}(\class{GP}))$ denotes)  the stable category of the Frobenius category of all chain complexes of Gorenstein injective (resp. Gorenstein projective) $R$-modules.
A lesser-known representation of  the stable category of complexes is 
$$\textnormal{StCh}(R)\simeq\textnormal{St}(\textnormal{C}(\class{GFC})),$$
 where $\textnormal{C}(\class{GFC})$ denotes  the Frobenius category of all chain complexes of $R$-modules that are \emph{Gorenstein flat-cotorsion} in the sense of~\cite{cet-totally-flat-cot-theory}.  
 As discussed in~\cite{gillespie-FP-pro-injective}, over a coherent ring $R$, the dual notion to that of a Gorenstein flat-cotorsion module is a \emph{Gorenstein FP-pro-injective} module (the definition is on page~\pageref{page-G-pro-inj}).   
 This leads us here to  a representation of $\textnormal{StCh}(R)$  that is dual to $\textnormal{St}(\textnormal{C}(\class{GFC}))$.   In particular, we show in  Theorem~\ref{theorem-FP-model}   that over a coherent ring $R$ we have  
$$\textnormal{St}(\textnormal{C}(\class{GFP})) \simeq\ \textnormal{StCh}(R)\simeq\textnormal{St}(\textnormal{C}(\class{GFC})), $$
where  now $\textnormal{C}(\class{GFP})$  denotes  the Frobenius category of all chain complexes of $R$-modules that are \emph{Gorenstein FP-pro-injective} in the sense of~\cite{gillespie-FP-pro-injective}.  Indeed   Theorem~\ref{theorem-FP-model} shows that the stable category, $\textnormal{St}(\textnormal{C}(\class{GFP}))$, is exactly the homotopy category of the model structure $\mathfrak{M}$, shown above in (\ref{eq-intro}),  of the previously mentioned  Theorem~\ref{them-cot-pair}.
The following statement summarizes  the main results of Section~\ref{section-Ding-Chen}.

\begin{theorem*}[see Proposition~\ref{prop-i-p-f-models}, Theorem~\ref{theorem-FP-model}, Corollary~\ref{cor-stable-cats}]
Let $R$ be a Ding-Chen ring. Then  $\textnormal{StCh}(R)$ exists,   it is compactly generated, and is triangle equivalent to each of the following stable categories:
$$\textnormal{St}(\textnormal{C}(\class{GI})) \simeq\ \textnormal{St}(\textnormal{C}(\class{GP})) \simeq\ \textnormal{St}(\textnormal{C}(\class{GFC}))\simeq\textnormal{St}(\textnormal{C}(\class{GFP})).$$
Here $\class{GI}$,  $\class{GP}$, $\class{GFC}$, and $\class{GFP}$ are respectively the classes of Gorenstein injective,  Gorenstein projective,  Gorenstein flat-cotorsion, and  Gorenstein FP-pro-injective modules. 
Each arises as the homotopy category of an abelian model structure on $\ch$, having $\tilclass{W}$ as its class of trivial objects, as follows: 
\begin{itemize}
\item The  injective model structure, $\mathfrak{M}_{inj} = (All, \tilclass{W},  \class{GI}_{\textnormal{C}})$. Its homotopy category is triangle equivalent to  $\textnormal{St}(\textnormal{C}(\class{GI}))$,   the stable category of the Frobenius category $\class{GI}_{\textnormal{C}} = \textnormal{C}(\class{GI})$ of all complexes of Gorenstein injectives.
\item The  projective model structure, $\mathfrak{M}_{proj} = ( \class{GP}_{\textnormal{C}}, \tilclass{W},  All)$.   Its homotopy category is triangle equivalent to  $\textnormal{St}(\textnormal{C}(\class{GP}))$,   the stable category  of the Frobenius category $\class{GP}_{\textnormal{C}} = \textnormal{C}(\class{GP})$ of all complexes of Gorenstein projectives.
\item The flat model structure, $\mathfrak{M}_{flat} = (\class{GF}_{\textnormal{C}}, \tilclass{W}, \textnormal{C}(Cot))$, where  $\textnormal{C}(Cot)$ is the class of all complexes of cotorsion modules. Its homotopy category is triangle equivalent to   $\textnormal{St}(\textnormal{C}(\class{GFC}))$, the stable category of the Frobenius category $\class{GF}_{\textnormal{C}} \mycap \textnormal{C}(Cot) = \textnormal{C}(\class{GFC})$.  Here, $\class{GFC}$  denotes  the class of all  Gorenstein flat-cotorsion modules in the sense of~\cite{cet-totally-flat-cot-theory}.   
\item The FP-injective model structure, $\mathfrak{M}_{fp} = (\textnormal{C}(\class{FP}), \tilclass{W},  ({}^\bot\textnormal{w}\mathcal{DI}_{\textnormal{C}})^\bot)$, where  $\textnormal{C}(\class{FP})$ is the class of all complexes of FP-projective modules.   Its homotopy category is triangle equivalent to $\textnormal{C}(\class{FP}) \mycap  ({}^\bot\textnormal{w}\mathcal{DI}_{\textnormal{C}})^\bot  = \textnormal{St}(\textnormal{C}(\class{GFP}))$, the stable category of the Frobenius category $\textnormal{C}(\class{GFP})$.  Here, $\class{GFP}$  denotes  the class of all  Gorenstein FP-pro-injective modules in the sense of~\cite{gillespie-FP-pro-injective}.   
\end{itemize}
\end{theorem*}

A key point behind the FP-injective model structure is that for a Ding-Chen ring $R$, the class of fibrant complexes satisfies
$$({}^\bot\textnormal{w}\mathcal{DI}_{\textnormal{C}})^\bot = \textnormal{C}(\class{Z}),$$
where $\class{Z}$ is the class of all cycle modules of exact complexes of FP-injective modules. It follows that this class is also both covering and special preenveloping in $\ch$; see Corollary~\ref{corollary-covers}.

\section{preliminaries}
 Foundational to this paper is the theory of cotorsion pairs,  covers, and envelopes  of modules (and chain complexes of modules) over a ring $R$. The basic theory and terminology we use can be found in~\cite{enochs-jenda-book} and~\cite{trlifaj-book}. We also use the book~\cite{gillespie-book} and the notation for complexes (such as $\tilclass{B}$ and $\dgclass{B}$ etc.) from~\cite{gillespie}.
 
Throughout the paper, $R$ is a (left) coherent ring unless explicitly stated otherwise.  $\rmod$ will denote  the category of (left) $R$-modules, and $\ch$ the associated category of chain complexes. 
Given to complexes $X, Y \in \ch$, the abelian group $\Hom_{\ch}(X,Y)$, of all chain maps $f \colon X \xrightarrow{} Y$, will simply be denoted by $\Hom(X,Y)$.

We recall that an $R$-module $A$ is called \emph{FP-injective} (or \emph{absolutely pure}) if $\Ext^1_R(F,A) = 0$ for all finitely presented modules $F$.  We use $\mathcal{FI}$ to denote the class of all FP-injective modules. Then by definition,  the  \emph{FP-projective} modules are those modules in $\leftperp{\mathcal{FI}}$, the left $\Ext$-orthogonal class to that of all FP-injective modules.  We let $\mathcal{FP} := \leftperp{\mathcal{FI}}$ denote this class of  modules. We have that $(\mathcal{FP}, \mathcal{FI})$ is a complete cotorsion pair, and $\mathcal{FP}$ equals the class of all direct summands of transfinite extensions of all fintely presented modules~\cite[Them.~3.8]{ding and mao 08}.  Since $R$ is coherent, the class of FP-injective modules is closed under cokernels of monomorphisms; see \cite[Prop. 4.2]{Kathy}. Thus the cotorsion pair $(\mathcal{FI}, \mathcal{FI})$ is also  hereditary.

The same definitions can be applied in $\ch$ to obtain \emph{FP-injective} and \emph{FP-projective} chain complexes. It is known that a complex is FP-injective if and only if it  it is in $\widetilde{\mathcal{FI}}$, the class of all exact complexes with each of its cycles in $\class{FI}$ \cite[\S3.2]{Ding-Chen-complex-models}.  Then letting $\textnormal{C}(\mathcal{FP})$ denote the class of all complexes with  components in $\mathcal{FP}$ we have the following.

\begin{lemma}\label{lemma-FP-cot-comp}
The FP-injective cotorsion pair in $\ch$ is precisely 
$$(\textnormal{C}(\mathcal{FP}),\tilclass{FI}),$$
and this is a complete hereditary cotorsion pair. 
\end{lemma}

\begin{proof}
The finitely presented complexes cogenerate a complete hereditary cotorsion pair, $(\dgclass{FP}, \tilclass{FI})$, by the proof of~\cite[Prop.~3.17]{Ding-Chen-complex-models}. But by 
\cite[Example 4.3]{saroch-stovicek-G-flat}, we have that $\dgclass{FP} = \textnormal{C}(\mathcal{FP})$. 
\end{proof}

\begin{definition}\label{def-weakly}
A complex $G$ is \emph{weakly Ding injective} if there is an exact complex of FP-injective complexes $\mathbb{E}$ such that $G = Z_0\mathbb{E}$ and with the property that $\Hom(F,\mathbb{E})$ is an exact complex of abelian groups for any FP-injective complex $F$.
\end{definition}

Definition~\ref{def-weakly} is the chain complex analog of the \emph{weakly Ding injective modules}; these were studied in~\cite{Goren-FP-inj, iacob-weakly-Ding-1, iacob-weakly-Ding-2, iacob-weakly-Ding-3, gillespie-FP-pro-injective}. If in Definition~\ref{def-weakly} we instead have that $\mathbb{E}$  is an exact complex of injective complexes (not just FP-injective complexes)  then $G$ is, by definition,  a \emph{Ding injective complex}. The same idea defines Ding injective modules. 
So it is immediate that any Ding injective complex (resp. module) is  a weakly Ding injective complex (resp. module). 
We will use the following notations throughout the paper. 
\begin{itemize}
\item $\mathcal{DI}$ denotes the class of all Ding injective modules.
\item $\textnormal{w}\class{DI}$ denotes the class of all weakly Ding injective modules. 
\item $\mathcal{DI}_{\textnormal{C}}$ denotes  the class of all Ding injective complexes. This coincides with the class of all complexes of  Ding injective modules;  see Lemma~\ref{lemma-DI-complexes}.
\item $\textnormal{w}\mathcal{DI}_{\textnormal{C}}$ denotes the class of all weakly Ding injective complexes. 
\item Given any class  of modules $\class{X}$, we let $\textnormal{C}(\mathcal{X})$ denote the class of all complexes of modules $X$ having each $X_n\in\class{X}$.
\end{itemize}

\begin{lemma}\cite[Main Theorem]{gillespie-iacob-Ding-inj-complexes}\label{lemma-DI-complexes}
For any ring $R$, a complex is Ding injective if and only if it is a complex of Ding injective modules. In our symbols, $\mathcal{DI}_{\textnormal{C}} = \textnormal{C}(\mathcal{DI})$.
\end{lemma}

\section{Results on weakly Ding injective complexes}\label{sec-results}
Again, $R$ is a (left) coherent ring throughout.

\begin{lemma}
The class $\widetilde{\mathcal{FI}}$ of all FP-injective complexes is covering in $\ch$.
    \end{lemma}
\begin{proof}
The class $\mathcal{FI}$ is closed under pure quotients and under pure submodules, so it is a Kaplansky class \cite[Prop.~3.2]{holm-jorgensen-precovers-purity}. Since $\mathcal{FI}$ is also closed under direct limits, it is a deconstructible class. By~\cite[Them.~4.2(2)]{stovicek-deconstructible}, the class  $\widetilde{\mathcal{FI}}$ of FP-injective complexes is deconstructible in $\ch$, therefore it is precovering. It is known that a precovering class that is closed under direct limits is covering~\cite[Them.~2.2.12]{xu_flat_covers}.
    \end{proof}

The arguments in~\cite[\S6]{gillespie-FP-pro-injective} extend to complexes. In particular, using Theorem~6.6 therein,  with the class $\mathcal{B}$ to be that of FP-injective complexes, we obtain:\\

\begin{proposition}\label{prop-weakly-Ding-sums}
  Any weakly Ding injective complex is equal to a direct sum $L \Oplus B$ with  $L$ a Ding injective complex, and with $B$ an FP-injective complex.
\end{proposition}

\begin{corollary}\label{cor-weak-degree}
Every weakly Ding injective complex is a complex of weakly Ding injective modules. In our symbols, $\textnormal{w}\mathcal{DI}_{\textnormal{C}} \subseteq \textnormal{C}(\textnormal{w}\class{DI})$. 
\end{corollary}

\begin{corollary}\label{corollary-left-perp}
The left orthogonal to the class of all weakly Ding injective complexes  is precisely the class of all complexes in  ${}^\bot \textnormal{C}(\mathcal{DI})$ having FP-projective components. I.e., 
$${}^\bot \textnormal{w}\mathcal{DI}_{\textnormal{C}} = \textnormal{C}(\mathcal{FP})\mycap {}^\bot \textnormal{C}(\mathcal{DI}).$$
\end{corollary}

\begin{proof}
 By combining Proposition~\ref{prop-weakly-Ding-sums} along with Lemmas~\ref{lemma-FP-cot-comp} and~\ref{lemma-DI-complexes}  
 we have  $^\bot \textnormal{w}\mathcal{DI}_{\textnormal{C}} = {}^\bot \textnormal{C}(\mathcal{DI}) \mycap  {}^\bot \widetilde{\mathcal{FI}}  =  {}^\bot \textnormal{C}(\mathcal{DI}) \mycap \textnormal{C}(\mathcal{FP})$.
\end{proof}

Next we will show that the left-right-perp to the class of all  weakly Ding injective complexes is nothing more than  the class of all complexes having components in the left-right-perp of the weakly Ding injective modules. That is, in our symbols, $$({}^\bot\textnormal{w}\mathcal{DI}_{\textnormal{C}})^\bot   = \textnormal{C}[({}^\bot\textnormal{w}\class{DI})^\bot].$$  
We will need some lemmas. 

\begin{lemma}\label{lemma-Hom-exact}  
If $X$ is a complex of modules with components in $(^\bot\textnormal{w}\mathcal{DI})^\bot$, then there is an exact and $\Hom (\widetilde{\mathcal{FI}},-)$ exact complex  
    $\cdots \rightarrow F_2 \rightarrow F_1 \rightarrow F_0 \rightarrow X \rightarrow 0$ with all of the $F_j$ being FP-injective complexes.
\end{lemma}

\begin{proof}
  Since $(^\bot \textnormal{C}(\mathcal{DI}),  \textnormal{C}(\mathcal{DI}) )$   is a complete hereditary cotorsion pair in $\ch$, there is a short  exact sequence $0 \rightarrow D \rightarrow F \rightarrow X \rightarrow 0$, with $D$ a Ding injective complex and with $F \in {}^\bot \textnormal{C}(\mathcal{DI})$.
  Then for each $n$, we have a short exact sequence $0 \rightarrow D_n \rightarrow F_n \rightarrow X_n \rightarrow 0$ with $D_n \in \class{DI} \subseteq (^\bot\textnormal{w}\mathcal{DI})^\bot$ and $X_n \in (^\bot\textnormal{w}\mathcal{DI})^\bot$, which gives us that $F_n$ is also in $(^\bot\textnormal{w}\mathcal{DI})^\bot$. But then $F_n \in {}^\bot\mathcal{DI}  \mycap ({}^\bot\textnormal{w}\mathcal{DI})^\bot$, and so  each $F_n$ is an FP-injective module by~\cite[Lemma~3]{iacob-weakly-Ding-3}. Therefore $F$ is an exact complex of FP-injective modules.
  
  Now if $I$ is any complex of injective modules, then $I \in \textnormal{C}(\mathcal{DI})$. Since  $F \in \leftperp{\textnormal{C}(\class{DI})}$ we  have $\Ext^1(F,I)=0$. Thus $F$ is a coacyclic complex of FP-injective modules. Since the ring is coherent, such a complex is FP-injective  by~\cite[Prop.~6.11]{stovicek-purity}.
  The complex $D$ is Ding injective, so there is an exact and $\Hom (\widetilde{\mathcal{FI}},-)$ exact complex $$\cdots \rightarrow F_2 \rightarrow F_1 \rightarrow D\rightarrow 0$$ with all $F_j$ injective complexes. Pasting it together with the short exact sequence $0 \rightarrow D \rightarrow F \rightarrow X \rightarrow 0$ gives us an exact and $\Hom (\widetilde{\mathcal{FI}},-)$ exact complex 
 $$\cdots \rightarrow F_2 \rightarrow F_1 \rightarrow F_0 \rightarrow X \rightarrow 0$$ with all of the $F_j$  being  FP-injective complexes.
\end{proof}

\begin{lemma}\label{lemma-dw-in-perp-perp}
   Every complex $X$ with each $X_n \in ({}^\bot\textnormal{w}\mathcal{DI})^\bot$ is in  $(^\bot\textnormal{w}\mathcal{DI}_\textnormal{C})^\bot$.
\end{lemma}

\begin{proof}
 Let $X$ be a complex of modules with each $X_n \in (^\bot\textnormal{w}\mathcal{DI})^\bot$. By the proof of Lemma~\ref{lemma-Hom-exact}, there is a short exact sequence $0 \rightarrow D \rightarrow F \rightarrow X \rightarrow 0$, with $D$ a Ding injective complex, and with $F$ an FP-injective complex.
 Let $Y \in {}^\bot\textnormal{w}\mathcal{DI}_{\textnormal{C}}= {}^\bot \textnormal{C}(\mathcal{DI}) \mycap \textnormal{C}(\mathcal{FP})$ (the equality is by Corollary~\ref{corollary-left-perp}). Then we obtain an exact sequence $0 = \Ext^1(Y,F) \rightarrow \Ext^1(Y,X) \rightarrow \Ext^2(Y,D)=0$.
  The first group here is zero by Lemma~\ref{lemma-FP-cot-comp}, and the last group is zero by Lemma~\ref{lemma-DI-complexes} along with the fact that the cotorsion pair $({}^\bot \textnormal{C}(\mathcal{DI}), \textnormal{C}(\mathcal{DI}))$ is hereditary. This shows $X \in ({}^\bot\textnormal{w}\mathcal{DI}_{\textnormal{C}})^\bot$.
\end{proof}

\begin{lemma}\label{lemma-perp-perp-in-dw}
If $X \in ({}^\bot\textnormal{w}\mathcal{DI}_{\textnormal{C}})^\bot$, then each $X_n$ is in $({}^\bot\textnormal{w}\mathcal{DI})^\bot$.
\end{lemma}

\begin{proof}
By Corollary~\ref{cor-weak-degree} we have  $\textnormal{w}\mathcal{DI}_{\textnormal{C}} \subseteq \textnormal{C}(\textnormal{w}\class{DI})$. 
Thus $\textnormal{w}\mathcal{DI}_{\textnormal{C}}  \subseteq \textnormal{C}[({}^\bot\textnormal{w}\class{DI})^\bot]$.
In general, if $(\class{X}, \class{Y})$ is a cotorsion pair in $\rmod$ then it is easy to check that $(\leftperp{\textnormal{C}(\class{Y})}, \textnormal{C}(\class{Y}))$ is a cotorsion pair in $\ch$.
So since $({}^\bot\textnormal{w}\class{DI},  ({}^\bot\textnormal{w}\class{DI})^\bot)$ is a cotorsion pair in  $\rmod$, it follows that $ ({}^\bot\textnormal{w}\mathcal{DI}_{\textnormal{C}})^\bot  \subseteq \textnormal{C}[({}^\bot\textnormal{w}\class{DI})^\bot]$.
\end{proof}

\begin{proposition}\label{prop-left-right-perp}
$({}^\bot\textnormal{w}\mathcal{DI}_{\textnormal{C}})^\bot   = \textnormal{C}[({}^\bot\textnormal{w}\class{DI})^\bot]$. \\
 In  words,   the left-right-perp to the class of all  weakly Ding injective complexes is precisely the class of all complexes having components in the left-right-perp to the class of all  weakly Ding injective modules. 
\end{proposition}

\begin{proof}
    By Lemma~\ref{lemma-dw-in-perp-perp} and Lemma~\ref{lemma-perp-perp-in-dw}.
\end{proof}

\begin{theorem}\label{them-cot-pair}
The class $\textnormal{w}\mathcal{DI}_{\textnormal{C}}$ of weakly Ding injective complexes generates an hereditary abelian model structure in  $\ch$, 
   $$\mathfrak{M} =  (\textnormal{C}(\mathcal{FP}) , {}^\bot \textnormal{C}(\mathcal{DI}) ,   \textnormal{C}[({}^\bot{\textnormal{w}\mathcal{DI}})^\bot]).$$
In particular, $(\textnormal{C}(\mathcal{FP}) \mycap {}^\bot \textnormal{C}(\mathcal{DI}) ,   \textnormal{C}[({}^\bot{\textnormal{w}\mathcal{DI}})^\bot])$ is a complete hereditary cotorsion pair. Moreover, we have $(^\bot\textnormal{w}\mathcal{DI}_{\textnormal{C}})^\bot = \textnormal{C}[(^\bot{\textnormal{w}\mathcal{DI}})^\bot]$, and a complex $Y$ is in this class if and only if there is a short exact sequence  $0 \rightarrow D \rightarrow F \rightarrow Y \rightarrow 0$, with $D$ a Ding injective complex and $F$ an FP-injective complex. 
\end{theorem}

\begin{proof}
This follows from~\cite[Thems.~3.2/3.4]{gillespie-FP-pro-injective}, along with what we've shown above. In more detail, the Ding injective complexes of Lemma~\ref{lemma-DI-complexes} are the right side of an injective cotorsion pair, $(\leftperp{\textnormal{C}(\class{DI})}, \textnormal{C}(\class{DI}))$, while the FP-injective cotorsion pair of 
Lemma~\ref{lemma-FP-cot-comp} is hereditary and has $\tilclass{FI} \subseteq \leftperp{\textnormal{C}(\class{DI})}$. So  \cite[Thems.~3.2/3.4]{gillespie-FP-pro-injective} gives us the model structue $\mathfrak{M}$, and the class of fibrant objects is as described because of  Proposition~\ref{prop-left-right-perp}    and Corollary~\ref{corollary-left-perp}.
\end{proof}

\subsection{Extension closed weakly Ding injectives}
By~\cite[Prop.~14]{iacob-weakly-Ding-3},  we have $\textnormal{w}\mathcal{DI} = ({}^\bot\textnormal{w}\mathcal{DI})^\bot$ whenever  the class of weakly Ding injective modules is closed under extensions. We know examples of such rings, but we don't know whether or not it holds for all rings. But for the case that it holds, we can summarize the previous results as follows. 

\begin{corollary}
 Assume that the class $\textnormal{w}\mathcal{DI}$ of weakly Ding injective modules is closed under extensions. Then each of the following hold:
 \begin{enumerate}
 \item  For any complex $X$ of weakly Ding injective modules,  there is an exact and $\Hom (\widetilde{\mathcal{FI}},-)$ exact complex   
    $\cdots \rightarrow F_2 \rightarrow F_1 \rightarrow F_0 \rightarrow X \rightarrow 0$  with all of the $F_j$ being FP-injective complexes.  
\item $({}^\bot\textnormal{w}\mathcal{DI}_{\textnormal{C}})^\bot = \textnormal{C}(\textnormal{w}\class{DI})$. \\
 In  words,   the left-right-perp to the class of all  weakly Ding injective complexes is merely the class of all complexes having weakly Ding injective components. 
 \item   $(^\bot\textnormal{w}\mathcal{DI}_{\textnormal{C}}, \textnormal{C}({\textnormal{w}\mathcal{DI}})) = 
 (\textnormal{C}(\mathcal{FP}) \mycap {}^\bot \textnormal{C}(\mathcal{DI}) ,   \textnormal{C}(\textnormal{w}\class{DI}))$ is a complete hereditary cotorsion pair in $\ch$, and we even have an abelian model structure, $$\mathfrak{M} =  (\textnormal{C}(\mathcal{FP}) , {}^\bot \textnormal{C}(\mathcal{DI}) ,   \textnormal{C}(\textnormal{w}\class{DI})).$$
\end{enumerate}
\end{corollary}

\begin{proof}
Statement (1) is by Lemma~\ref{lemma-Hom-exact}; statement (2) is by Proposition~\ref{prop-left-right-perp}; and statement (3) is by Theorem~\ref{them-cot-pair}.
\end{proof}

\section{The stable chain complex category of a Ding-Chen ring}\label{section-Ding-Chen}

Throughout this section, unless explicitly stated otherwise, we let $R$ be  a Ding-Chen ring. 
We recall that this means $R$  is a two-sided coherent ring with finite self FP-injective dimensions,   $FP\textnormal{-id}(R_R)$ and $FP\textnormal{-id}({}_RR)$.  Ding and Chen showed  in~\cite[Cor.~3.18]{ding and chen 93} that such rings satisfy  $FP\textnormal{-id}(R_R) = FP\textnormal{-id}({}_RR) = d$ for some non-negative integer $d$. 
Moreover, they show that the following statements about the flat and  FP-injective dimensions of any given  $R$-module $M$ are equivalent: 
$$\circ \,  \textnormal{fd}(M)  < \infty \ \ \ \ \  \circ \,  \textnormal{fd}(M)  \leq d    \ \ \ \ \    \circ \,   \textnormal{FP-id}(M) < \infty  \ \ \ \ \   \circ \,   \textnormal{FP-id}(M) \leq d.$$
We will let $\class{W}$ denote the class of all such modules $M$. This is precisely the class of trivial objects in the Gorenstein projective (resp. injective, resp. flat, resp. FP-injective) model structures for the stable module category of $R$. See~\cite{gillespie-FP-pro-injective}.

\begin{notation}\label{notation-relative-G-inj}
We will  use the following facts and notation throughout this section. The statements asserted in (2), (3), and (4) can be found in~\cite[Them~1.1]{gillespie-ding-modules}.
\begin{enumerate}
\item $\class{Z}$ denotes the class of all modules that are equal to a cycle of some exact complex of FP-injective modules. We have $\class{Z} = \rightperp{(\leftperp{\textnormal{w}\class{DI}})}$, by~\cite[\S8]{gillespie-FP-pro-injective}.
\item $\class{GI} = \class{DI}$ denotes the class of all Gorenstein injective (= Ding injective) modules. These are nothing more than the  modules that are  equal to a cycle module of  some exact complex of  injective modules. 
\item $\class{GP} = \class{DP}$ denotes the class of all Gorenstein projective (= Ding projective) modules. These are nothing more than the  modules that are equal to a cycle module of some exact complex of projective modules. 
\item $\class{GF} = \class{DF}$ denotes the class of all Gorenstein flat (= Ding  flat) modules. These are nothing more than the modules that are equal to a cycle module of some exact complex of flat modules. 
\item $\tilclass{W}$ denotes the class of all exact complexes with cycles in $\class{W}$. It is not difficult to see that this coincides with the class of all complexes having finite flat dimension, equivalently, finite FP-injective dimension. 
\end{enumerate}
\end{notation}

The next lemma is that statements (2)--(4) above have analogs for chain complexes. The analog of (1) will appear in Theorem~\ref{theorem-FP-model}(1).

\begin{lemma}\label{lemma-Goren-complexes}
Let $R$ be Ding-Chen. Then a Ding injective (resp. projective, resp. flat)  chain complex is nothing more than a Gorenstein injective  (resp. projective, resp. flat) complex, 
 and these are nothing more than complexes that are degreewise Gorenstein injective (resp. projective, resp. flat) modules. 
This is summarized  in our notation as
\begin{enumerate}\setcounter{enumi}{5}
\item $\class{DI}_{\textnormal{C}} = \class{GI}_{\textnormal{C}} = \textnormal{C}(\class{GI}) =  \textnormal{C}(\class{DI})$. \\
\item $\class{DP}_{\textnormal{C}} = \class{GP}_{\textnormal{C}} = \textnormal{C}(\class{GP}) =  \textnormal{C}(\class{DP})$. \\
\item $\class{DF}_{\textnormal{C}} = \class{GF}_{\textnormal{C}} = \textnormal{C}(\class{GF}) =  \textnormal{C}(\class{DF})$.
\end{enumerate}
\end{lemma}

\begin{proof}
This follows from~\cite[Theorems~1.1/1.2]{gillespie-ding-modules} 
along with~\cite[Theorem~2.2]{yang-liu-gorenstein-complexes}, or~\cite[Corollaries~8.1 and 8.3]{gillespie-recollement}. 
\end{proof}

\begin{proposition}\label{prop-i-p-f-models}
Let $R$ be Ding-Chen. 
We have the following abelian model structures on $\ch$, each having $\tilclass{W}$ as its class of trivial objects, and with (co)fibrant complexes as described in Lemma~\ref{lemma-Goren-complexes}:
\begin{itemize}
\item The  injective model structure, $\mathfrak{M}_{inj} = (All, \tilclass{W},  \class{GI}_{\textnormal{C}})$. Its homotopy category is triangle equivalent to  $\textnormal{St}(\textnormal{C}(\class{GI}))$,   the stable category of the Frobenius category $\class{GI}_{\textnormal{C}} = \textnormal{C}(\class{GI})$. 
\item The  projective model structure, $\mathfrak{M}_{proj} = ( \class{GP}_{\textnormal{C}}, \tilclass{W},  All)$.   Its homotopy category is triangle equivalent to  $\textnormal{St}(\textnormal{C}(\class{GP}))$,   the stable category  of the Frobenius category $\class{GP}_{\textnormal{C}} = \textnormal{C}(\class{GP})$.
\item The flat model structure, $\mathfrak{M}_{flat} = (\class{GF}_{\textnormal{C}}, \tilclass{W}, \textnormal{C}(Cot))$, where  $\textnormal{C}(Cot)$ is the class of all complexes of cotorsion modules. Its homotopy category is triangle equivalent to   $\textnormal{St}(\textnormal{C}(\class{GFC}))$, the stable category of the Frobenius category $\class{GF}_{\textnormal{C}} \mycap \textnormal{C}(Cot) = \textnormal{C}(\class{GFC})$.   Here, $\class{GFC}$  denotes  the class of all  Gorenstein flat-cotorsion modules in the sense of~\cite{cet-totally-flat-cot-theory}.   
\end{itemize}
\end{proposition}

\begin{proof}
The existence of the three model structures follows from the work of Yang, Liu, and Liang in~\cite[\S4]{Ding-Chen-complex-models}, combined with the statements in Lemma~\ref{lemma-Goren-complexes}.  Each of the three model structures are hereditary, so by a standard result, the full subcategory of bifibrant (i.e., cofibrant and fibrant) objects is a Frobenius category whose projective-injective objects are exactly the trivially bifibrant objects. For the injective (resp. projective, resp. flat) model structure these are precisely the contractible complexes of   injective (resp. projective, resp. flat-cotorsion) modules. 

In the case of  $\mathfrak{M}_{flat} = (\class{GF}_{\textnormal{C}}, \tilclass{W}, \textnormal{C}(Cot))$,  the bifibrant objects satisfy 
$$\class{GF}_{\textnormal{C}}\mycap\textnormal{C}(Cot) = \textnormal{C}(\class{GF})\mycap\textnormal{C}(Cot) = \textnormal{C}(\class{GF}\mycap Cot) = \textnormal{C}(\class{GFC}),$$ where 
 the last equality is by~\cite[\S4 and \S5]{cet-totally-flat-cot-theory}.   So the homotopy category of $\mathfrak{M}_{flat}$ is equivalent to $\textnormal{St}(\textnormal{C}(\class{GFC}))$, the stable category of  $\textnormal{C}(\class{GFC})$.       
\end{proof}

As in~\cite{gillespie-FP-pro-injective} we say that an $R$-module (here $R$ may be any ring) is \textbf{\emph{FP-pro-injective}} if it is both FP-projective and FP-injective. Then we say that a module $M$ is  \textbf{\emph{Gorenstein FP-pro-injective}}\label{page-G-pro-inj} if  $M = Z_0X$ for some exact complex $X$ of FP-pro-injective modules having the property that both  $\Hom_R(N, X)$ and $\Hom_R(X, N)$  are also  exact whenever $N$ is another FP-pro-injective. In fact, the exactness of  $\Hom_R(X, N)$ is automatic for (left) coherent rings due to the periodicity of the class of FP-projectives; \cite[Example 4.3]{saroch-stovicek-G-flat}. Gorenstein FP-pro-injective modules also appeared  in~\cite[\S4]{cet-one-sided-Goren}  under the name of \emph{Gorenstein fp-injective} modules.

We now show that in the current case of a Ding-Chen ring $R$, 
 the abelian model structure of Theorem~\ref{them-cot-pair} yields an  FP-injective analog to go along with those in Proposition~\ref{prop-i-p-f-models}.   In particular, this one is dual to the flat model structure.

\begin{theorem}\label{theorem-FP-model}
Let $R$ be Ding-Chen. 
The class $\textnormal{w}\mathcal{DI}_{\textnormal{C}}$ of all weakly Ding injective complexes generates what we call the \textbf{FP-injective model structure} on $\ch$, 
$$\mathfrak{M}_{fp} = (\textnormal{C}(\class{FP}), \tilclass{W}, ({}^\bot{\textnormal{w}\mathcal{DI}}_{\textnormal{C}})^\bot).$$   
The class $\textnormal{C}(\class{FP})$ of cofibrant objects is the class of all complexes of FP-projective modules, and  the class $({}^\bot{\textnormal{w}\mathcal{DI}}_{\textnormal{C}})^\bot$ of fibrant objects has the following easier descriptions:
\begin{enumerate}
\item  $({}^\bot{\textnormal{w}\mathcal{DI}}_{\textnormal{C}})^\bot = \textnormal{C}(\class{Z})$,  the class of all complexes of modules in $\class{Z}$.
\item  $({}^\bot{\textnormal{w}\mathcal{DI}}_{\textnormal{C}})^\bot$ is the class of all complexes $X$ fitting into a short exact sequence $$0 \rightarrow G \rightarrow F \rightarrow X \rightarrow 0$$ where $G$ is a Gorenstein  injective complex  
and $F$ is an FP-injective complex. 
\end{enumerate}
The homotopy category of $\mathfrak{M}_{fp} $ is triangle equivalent to  $\textnormal{St}(\textnormal{C}(\class{GFP}))$, the stable category of the Frobenius category $\textnormal{C}(\class{FP})\mycap \textnormal{C}(\class{Z})  = \textnormal{C}(\class{GFP})$.
 Here, $\class{GFP}$  denotes  the class of all  Gorenstein FP-pro-injective modules in the sense of~\cite[\S8]{gillespie-FP-pro-injective}.   
\end{theorem}

\begin{proof}
By Lemma~\ref{lemma-Goren-complexes} and Proposition~\ref{prop-i-p-f-models}, we have $\tilclass{W} = {}^\bot \textnormal{C}(\mathcal{GI}) = {}^\bot \textnormal{C}(\mathcal{DI})$. So by Theorem~\ref{them-cot-pair}, the class $\textnormal{w}\mathcal{DI}_{\textnormal{C}}$ of weakly Ding injective complexes generates an hereditary abelian model structure on  $\ch$, 
   $$\mathfrak{M}_{fp} =  (\textnormal{C}(\class{FP}) ,  \tilclass{W} ,   ({}^\bot{\textnormal{w}\mathcal{DI}}_{\textnormal{C}})^\bot),$$
 having  $(^\bot\textnormal{w}\mathcal{DI}_{\textnormal{C}})^\bot = \textnormal{C}((^\bot{\textnormal{w}\mathcal{DI}})^\bot)$. But we have $(^\bot{\textnormal{w}\mathcal{DI}})^\bot = \class{Z}$, by~\cite[Prop.~8.1]{gillespie-FP-pro-injective}. So the class of fibrant objects is nothing more than  the class $\textnormal{C}(\class{Z})$ of all complexes of modules in $\class{Z}$. 
Therefore,  the Frobenius category of all  bifibrant objects is  (see~\cite[Prop. 8.3(1)]{gillespie-book}) 
$$\textnormal{C}(\class{FP}) \mycap  ({}^\bot{\textnormal{w}\mathcal{DI}}_{\textnormal{C}})^\bot  = \textnormal{C}(\class{FP})\mycap \textnormal{C}(\class{Z})  
= \textnormal{C}(\class{FP}\mycap \class{Z}) = \textnormal{C}(\class{GFP}),$$  where 
 the last equality is by~\cite[\S8]{gillespie-FP-pro-injective}.   So the homotopy category of $\mathfrak{M}_{fp}$ is triangle equivalent to $\textnormal{St}(\textnormal{C}(\class{GFP}))$, the stable category of  $\textnormal{C}(\class{GFP})$.  (Note: The projective-injective objects are characterized in Proposition~\ref{prop-pro-inj}.) 
 
 The remaining characterization of the fibrant objects  follows immediately  from Theorem~\ref{them-cot-pair} and Lemma~\ref{lemma-Goren-complexes}.
\end{proof}

The next proposition characterizes the projective-injective objects in the  Frobenius category $\textnormal{C}(\class{GFP})$.  

\begin{proposition}\label{prop-pro-inj}
Let $R$ be a  Ding-Chen ring and $X$ be a chain complex of $R$-modules. 
Then the following statements are equivalent:
\begin{enumerate}
\item $X$ is an \textbf{FP-pro-injective} complex. I.e., $X \in \textnormal{C}(\class{FP})\mycap \tilclass{FI}$, the core of the FP-injective cotorsion pair, $(\textnormal{C}(\class{FP}), \tilclass{FI})$, of Lemma~\ref{lemma-FP-cot-comp}.
 \item  $X$ is a projective-injective object of $\textnormal{C}(\class{GFP})$, the Frobenius category of all complexes of FP-pro-injective modules.
\item $X \in \tilclass{FP\cap FI}$, the class of all exact complexes having FP-pro-injective cycle modules.
\item $X$ is a contractible complex of FP-pro-injective modules. 
 \end{enumerate}
 \end{proposition}

\begin{proof}
By ~\cite[Prop. 8.3(1)]{gillespie-book}, the class of projective-injective objects in the Frobenius category $\textnormal{C}(\class{FP})\mycap \textnormal{C}(\class{Z}) =  \textnormal{C}(\class{GFP})$  is  precisely  the  core of the  model structure  $\mathfrak{M}_{fp} =  (\textnormal{C}(\class{FP}) ,  \tilclass{W} ,   \textnormal{C}(\class{Z}))$ from  Theorem~\ref{theorem-FP-model}. 
That is, the class
$$\textnormal{C}(\class{FP}) \mycap  \tilclass{W} \mycap \textnormal{C}(\class{Z})   = 
\textnormal{C}(\class{FP}) \mycap  \tilclass{FI} =  \tilclass{FP\cap FI},$$ 
where the last equality follows from  the periodicity of the class of FP-projectives as shown in \cite[Example 4.3]{saroch-stovicek-G-flat}.
So statements (1) and (2) and (3) are equivalent. But note that exact complexes with FP-pro-injective cycles are necessarily split exact complexes,  hence contractible complexes of FP-pro-injective modules.   So (4) is also equivalent. 
\end{proof}

\subsection{The stable chain complex category of $R$}
Recall that the \emph{stable module category} of a Ding-Chen ring $R$  is defined by 
$$\textnormal{Stmod}(R) := \rmod/\class{W},$$
where the right hand side denotes the triangulated localization (in the sense of~\cite[\S6.7]{gillespie-book}) of the category of $R$-modules by the full subcategory of all modules of finite flat dimension, equivalently, of finite FP-injective dimension. 
In the same way, we define the \textbf{\emph{stable chain complex category}} of $R$ to be the triangulated localization  
$$\textnormal{StCh}(R) := \ch/\tilclass{W}.$$
Then analogous to what was shown in~\cite{gillespie-FP-pro-injective} we have  now shown the following result.  

\begin{corollary}\label{cor-stable-cats}
Let $R$ be a Ding-Chen ring. Then  $\textnormal{StCh}(R)$ exists,   it is compactly generated, and it is triangle equivalent to each of the following stable categories:
$$\textnormal{St}(\textnormal{C}(\class{GI})) \simeq\ \textnormal{St}(\textnormal{C}(\class{GP})) \simeq\ \textnormal{St}(\textnormal{C}(\class{GFC}))\simeq\textnormal{St}(\textnormal{C}(\class{GFP})).$$
Here $\class{GI}$,  $\class{GP}$, $\class{GFC}$, and $\class{GFP}$ are respectively the classes of Gorenstein injective,  Gorenstein projective,  Gorenstein flat-cotorsion, and  Gorenstein FP-pro-injective modules. 
\end{corollary} 

\begin{proof}
The existence of $\textnormal{StCh}(R)$ and that it is triangle equivalent to the four stable categories is immediate from Proposition~\ref{prop-i-p-f-models} and Theorem~\ref{theorem-FP-model} and~\cite[\S8.2]{gillespie-book}. 
We note that each of the respective categories of complexes of Gorenstein modules is Frobenius with its projective-injective objects being exactly the contractible complexes of injective (resp. projective, resp.  flat-cotorsion, resp. FP-pro-injective) modules.

It is left to see why these categories are compactly generated. Using~\cite[Cor.~8.5]{gillespie-book}, this  follows from the fact that the cotorsion pairs associated to the Gorenstein projective model structure, that is, $(\textnormal{C}(\class{GP}),\tilclass{W})$ and   $(\tilclass{P}, All)$ are both cogenerated by sets of finitely presented objects (complexes). The canonical projective cotorsion pair,  $(\tilclass{P}, All)$, is cogenerated by the set $\{D^n(R)\}$, and these are certainly finitely presented. As for   $(\textnormal{C}(\class{GP}),\tilclass{W})$, the arguments from~\cite[Them.~8.3]{hovey} and~\cite{gillespie-AC-Gorenstein-stable} generalize from modules to complexes. We explain this briefly for the interested reader:
First, a chain complex $F$ is finitely presented in $\ch$ if and only if it is a  bounded complex of finitely presented modules. Since $R$ is coherent, we may find for any such $F$ a resolution  
$$\cdots \rightarrow{} P_2  \rightarrow{} P_1  \rightarrow{} P_0  \rightarrow{} F  \rightarrow{} 0$$ with each $P_i$ a finitely generated projective complex. Let $\Omega^1F = \Ker(P_{0} \rightarrow{} F)$ denote the first syzygy, and more generally let $\Omega^nF = \Ker(P_{n-1} \rightarrow{} P_{n-2})$ denote the $n$-th syzygy. Each $\Omega^nF$ is also finitely presented, again because $R$ is coherent.
Letting $d = FP\textnormal{-id}(R_R) = FP\textnormal{-id}({}_RR)$ be the dimension of the Ding-Chen ring, $R$, we take a set $\class{S} = \{\Omega^dF\}$ of  $d$-th syzygies, where $F$ ranges through the set of all (isomorphism representatives of) finitely presented complexes.
Then a dimension shifting argument shows that $\rightperp{\class{S}}$ equals the class of all complexes  $X$ satisfying  $FP\textnormal{-id}(X) \leq d$. It follows that $\rightperp{\class{S}} = \tilclass{W}$, and so $\class{S}$ is a set of finitely presented complexes that cogenerates the Gorenstein projective cotorsion pair. 
\end{proof}

\subsection{$\textnormal{C}(\class{Z})$-covers and preenvelopes}
Finally, we note  that $\textnormal{C}(\class{Z})$ is both covering and enveloping in $\ch$. Indeed it is a special preenveloping class by Theorem~\ref{theorem-FP-model}.  The fact that it is also covering  is immediate from the following result.  

\begin{theorem}\cite[Them.~5]{iacob-weakly-Ding-1}
  Let $R$ be a coherent ring  with the property
that there is a nonnegative integer $n$ such that every injective $R$-module has flat dimension $\le n$. Then the class $\textnormal{C}(\class{Z})$ is both covering and preenveloping in the category of complexes of $R$-modules.
\end{theorem}

\begin{corollary}\label{corollary-covers}
  Let $R$ be a Ding-Chen ring.   Then the class $\textnormal{C}(\class{Z})$ is both covering and special preenveloping in the category of complexes of $R$-modules.
\end{corollary}



\begin{thebibliography}{9}



\bibitem[BHP24]{bazzoni-hrbek-positselski-fp-projective-periodicity} Silvana Bazzoni, Michal Hrbek, and Leonid Positselski, \emph{fp-Projective periodicity},  J. Pure Appl. Algebra vol.~228, no.~3, 2024.


\bibitem[CET20]{cet-totally-flat-cot-theory} Lars Winther Christensen, Sergio Estrada, and Peder Thompson, \emph{Homotopy categories of totally acyclic complexes with applications to the flat–cotorsion theory}, Categorical, homological and combinatorial methods in algebra, Contemp. Math. vol. 751, Amer. Math. Soc., Providence, RI, 2020, pp.~99--118.


\bibitem[CET24]{cet-one-sided-Goren} 
Lars Winther Christensen, Sergio Estrada, and Peder Thompson,  \emph{One-sided Gorenstein rings},  Forum Math.  vol.~36, no.~5, 2024, pp.~1411--1433. https://doi.org/10.1515/forum-2023-0303

\bibitem[DC93]{ding and chen 93} N. Ding and J. Chen, \emph{The flat dimensions of injective modules}, Manuscripta Math. vol.~78, no.~2, 1993, pp.~165--177.

\bibitem[DM08]{ding and mao 08}
N. Ding and L. Mao, \emph{Gorenstein FP-injective and Gorenstein flat modules}, J. Algebra Appl. vol.~7, no.~4, 2008, pp.~491-506.
 
 \bibitem[EJT93]{G-flat-modules}
E.E. Enochs, O.M.G. Jenda and B. Torrecillas.
\emph{Gorenstein flat modules}, Nanjing Univ. J. Math. Biquarterly vol.~10, no.~1, 1993, pp.~1--9. 
 
 \bibitem[EJ00]{enochs-jenda-book} E.~Enochs and O.~Jenda, \emph{Relative homological algebra}, de  Gruyter Exp. Math. vol. 30, Walter de Gruyter, New York, 2000.

 \bibitem[Gil04]{gillespie} James Gillespie, \emph{The flat model structure on Ch(R)}, Trans. Amer. Math. Soc. vol.~356, no.~8, 2004, pp.~3369--3390.

 \bibitem[Gil10]{gillespie-Ding-Chen rings} 
 James Gillespie, \emph{Model structures on modules over Ding-Chen rings}, Homology, Homotopy Appl. vol.~12, no.~1, 2010, pp.~61--73.
 
 \bibitem[Gil16]{gillespie-recollement}
James Gillespie, \emph{Gorenstein complexes and recollements from cotorsion  pairs}, Advances in Mathematics vol.~291, 2016, pp.~859--911.


\bibitem[Gil17]{gillespie-ding-modules} James Gillespie, \emph{On Ding injective, Ding projective and Ding flat modules and complexes}, Rocky Mountain J. Math. vol.~47, no.~8, 2017, pp.~2641--2673.

 \bibitem[Gil19]{gillespie-AC-Gorenstein-stable}
  James Gillespie, \emph{AC-Gorenstein rings and their stable module categories}, Journal of the Australian Mathematical Society vol.~107, no.~2, 2019,  pp.~181--198.


\bibitem[Gil24]{gillespie-book} James, Gillespie,  \emph{Abelian Model Category Theory}, Cambridge University Press; 2024.

\bibitem[Gil26]{gillespie-FP-pro-injective} James Gillespie, \emph{Acyclic complexes of FP-injective modules over Ding-Chen rings}, arXiv:2602.09371v1.


\bibitem[GI23]{gillespie-iacob-Ding-inj-complexes} James Gillespie and Alina Iacob, \emph{Ding injective envelopes in the category of complexes},  Rend. Circ. Mat. Palermo, II. Ser vol.~72,  2023, pp.~997--1004. 


\bibitem[GT06]{trlifaj-book} R{\"u}diger G{\"o}bel and Jan Trlifaj, \emph{Approximations and Endomorphism Algebras of Modules}, de Gruyter Exp. Math. vol.~41, Walter de Gruyter \& Co.,  Berlin, 2006.

\bibitem[HJ08]{holm-jorgensen-precovers-purity} Henrik Holm and Peter J\o rgensen, \emph{Covers, precovers, and purity}, Illinois J. Math. vol.~52,  no.~2, 2008, pp.~691–-703.

\bibitem[Hov02]{hovey} Mark Hovey, \emph{Cotorsion pairs, model category structures,  and representation theory}, Math. Z. vol.~241,  553-592, 2002.

\bibitem[Iac23]{iacob-weakly-Ding-1} Alina Iacob, \emph{Weakly Ding injective modules and complexes}, Comm. Algebra. vol.~51, no.~12, 2023, pp.~4899--4912.

\bibitem[Iac25]{iacob-weakly-Ding-2} Alina Iacob, \emph{Weakly Ding injective preenvelopes and covers}, Comm. Algebra. vol.~53, no.~10, 2025, pp.~4487--4497.

\bibitem[Iac26]{iacob-weakly-Ding-3} Alina Iacob, \emph{Gorenstein flat preenvelopes and Weakly Ding covers},  arXiv:2601.15469v1,  2026.

\bibitem[Pin05]{Kathy}
K. Pinzon, \emph{Absolutely pure modules}, University of Kentucky, 
\newblock  PhD Dissertation, 2005.


\bibitem[Ste70]{stenstrom-fp} Bo~Stenstr{\"o}m, \emph{Coherent rings and FP-injective modules}, J.  London Math. Soc. (2) vol.~2, 1970, pp.~323--329.

\bibitem[Sto13]{stovicek-deconstructible} Jan {\v{S}}\v{t}ov{\'{\i}}{\v{c}}ek, \emph{Deconstructibility and the {H}ill lemma  in {G}rothendieck categories}, Forum Math. vol.~25, no.~1, 2013,  pp.~193--219.
  
\bibitem[Sto14]{stovicek-purity} Jan {\v{S}}\v{t}ov{\'{\i}}{\v{c}}ek, \emph{On purity and applications to coderived and singularity categories},  arXiv:1412.1615.


\bibitem[SS20]{saroch-stovicek-G-flat} Jan {\v{S}}aroch and Jan {\v{S}}\v{t}ov{\'{\i}}{\v{c}}ek, \emph{Singular compactness and definability for $\Sigma$-cotorsion and Gorenstein modules}, Selecta Math. (N.S.) vol.~26, no. 2, 2020, Paper No. 23, 40 pp. 


\bibitem[Xu96]{xu_flat_covers}
  Jinzhong Xu, \emph{Flat Covers of Modules}, Lecture Notes in Mathematics, vol.~1634, Springer-Verlag, Berlin – Heidelberg, 1996.
  
\bibitem[YLL13]{Ding-Chen-complex-models}
Gang Yang, Zhongkui Liu, and Li Liang, \emph{Model structures on categories of complexes over Ding-Chen rings}, Communications in Algebra, vol.~41, 2013, pp.~50--69.


\bibitem[YL11]{yang-liu-gorenstein-complexes}
Xiaoyan Yang and Zhongkui Liu, \emph{Gorenstein projective, injective, and flat complexes}, Comm. Algebra. vol.~39, no.~5, 2011, pp.~1705--1721.


\bibitem[ZC11]{Goren-FP-inj} Yuedi Zeng and Jianlong Chen, \emph{On Gorenstein FP-injective modules},  J. Southeast Univ. (English Ed.) vol.~27, no.~1,  2011, pp.~115--118.


 \end{thebibliography}
\end{document}